\documentclass[pdflatex]{sn-jnl}% Math and Physical Sciences Numbered Reference Style  

\usepackage{amsmath,amssymb,amsfonts,amsthm}
\usepackage{graphicx}%
\usepackage{multirow}%
\usepackage{lmodern}
\usepackage[title]{appendix}%
\usepackage{xcolor}%
\usepackage{textcomp}%
\usepackage{manyfoot}%
\usepackage{booktabs}%  
\usepackage{listings}%
\usepackage{upgreek}
\usepackage{xspace}
\usepackage{mathtools}%
\usepackage{enumerate}%
\usepackage[english]{babel}

\newcommand{\eps}{\ensuremath{\varepsilon}}



\newcommand{\dx}{\mathrm{d}}

\definecolor{carmine}{rgb}{0.7, 0.1, 0.09}

\definecolor{debianred}{rgb}{0.84, 0.04, 0.33}
\newtheoremstyle{sltheorems}% name
{10pt}%  Space above
{6pt}%  Space below
{\sl}% Body font
{}% Indent amount (empty = no indent, \parindent = para indent)
{\bfseries}% Thm head font
{.}%Punctuation after thm head
{.5em}% Space after thm head: " " = normal interword space;
{\thmname{#1}\thmnumber{~#2}\thmnote{~(#3)}}
\theoremstyle{sltheorems} 
\newtheorem{Thm}{Theorem}
\newtheorem*{Thm*}{Theorem}
\newtheorem{CThm}{Classical Theorem}

\newtheorem{hypothesis}{Hypothesis}

\newtheorem{lem}{Lemma}

\newtheorem*{cor*}{Corollary}

\newtheoremstyle{plainrem}% name
{10pt}%  Space above
{6pt}%  Space below
{}% Body font
{}% Indent amount (empty = no indent, \parindent = para indent)
{\bfseries}% Thm head font
{.}%Punctuation after thm head
{.5em}% Space after thm head: " " = normal interword space;
{\thmname{#1}\thmnumber{~#2}\thmnote{~(#3)}}
\theoremstyle{plainrem}
\newtheorem{Rem}{Remark}

\theoremstyle{definition}

\usepackage{ifthen} 
\newboolean{alefonts}
\setboolean{alefonts}{false}
\ifthenelse{\boolean{alefonts}}{
\usepackage[full]{textcomp} 
\usepackage{newtxtext} % osf for text, not math
\usepackage{cabin} % sans serif 
\usepackage{zlmtt}% latin modern typewriter
\usepackage[bigdelims,vvarbb]{newtxmath} % bb from STIX
\usepackage[cal=boondoxo]{mathalfa} % mathcal
 }

\makeatletter
\patchcmd{\env@cases}{1.2}{1}{}{}
\makeatother

\begin{document}

\title[]{Beyond the Riemann Hypothesis bounds: A pair-correlation approach to the least prime in arithmetic progression
and the smallest quadratic non-residue}

%\author{Neelam Kandhil, Alessandro Languasco, Pieter Moree}

\author*[1]{\fnm{Neelam} \sur{Kandhil}}\email{\small{neelamkandhil091@gmail.com}}
\equalcont{\small{These authors contributed equally to this work.}}
\author[2]{\fnm{Alessandro} \sur{Languasco}}\email{\small{alessandro.languasco@unipd.it}}
\equalcont{\small{These authors contributed equally to this work.}}
\author[3]{\fnm{Pieter} \sur{Moree}}\email{\small{moree@mpim-bonn.mpg.de}}
\equalcont{\small{These authors contributed equally to this work.}}

\affil[1,3]{\small{\orgdiv{Max-Planck-Institut f\"ur Mathematik}, 
\orgaddress{\street{Vivatsgasse 7}, \city{Bonn}, \postcode{D-53111}, \country{Germany}}}}

\affil[2]{\small{\orgdiv{Dipartimento di Ingegneria dell'Informazione - DEI}, \orgname{Universit\`a di Padova},
\orgaddress{\street{Via Gradenigo 6/b}, \city{Padova}, \postcode{35131}, \country{Italy}}}}

\abstract{
The Generalized Riemann Hypothesis (GRH) has long defined the expected bounds for the 
smallest prime in an arithmetic progression and the least quadratic non-residue. However, this hypothesis primarily addresses the horizontal location of non-trivial zeros. In this paper, we show that incorporating the vertical spacing--or pair-correlation--of these zeros allows us to surpass these classical bounds. 
By combining these two zero-distribution perspectives, we establish
sharper estimates for both problems under GRH and specific pair-correlation
hypotheses, thereby providing a new link between pair-correlation 
phenomena  for Dirichlet 
$L$-functions and these two classical problems.}


%11L40 Estimates on character sums
%11N13 Primes in congruence classes
%11M50 Relations with random matrices
%11M26 Nonreal zeros of ζ(s) and L(s, χ); Riemann and other hypotheses

\maketitle

\section{Introduction}
The distribution of primes and the behavior of arithmetic characters are central themes in analytic number theory. Two of the most celebrated problems in these 
areas concern the size of the smallest prime in an arithmetic progression and 
the least quadratic non-residue. Historically, these problems are often associated with the names of Linnik and Vinogradov and the progress on these problems has been dictated by our
understanding of the distribution of the non-trivial zeros of Dirichlet $L$-functions. 

While the pole of Riemann zeta function at $s=1$ determines the ``main term'' of our counting functions, the location of zeros of Dirichlet $L$-functions governs the size of the ``error terms'' that ultimately define these bounds (see Lemma \ref{davenlem} 
in \S~\ref{pre}). 
Under the assumption of the Generalized Riemann Hypothesis (GRH), we obtain a rigid horizontal control that places all non-trivial zeros on the critical line.
In the classical approach, while we assume the zeros lie on the critical line, we have very little information about how they are spaced vertically. 
Recent developments in the spectral theory of these zeros, most notably the Montgomery pair-correlation conjecture (\cite{Montgomery1973}) for the Riemann zeta-function, suggest that these zeros exhibit a form of repulsion that leads to a much more regular and even   vertical distribution than previously assumed. This phenomenon reflects a deep 
connection with Random Matrix Theory, suggesting that the spacing of zeros of zeta functions mimics the distribution of eigenvalues of a generic large unitary matrix after appropriate scaling. In our previous work \cite{mathannalen}, we have shown 
that a suitable analogue of the Montgomery pair-correlation conjecture formulated for the non-trivial zeros of Dirichlet $L$-functions
has very strong implications for the distribution of primes in arithmetic progressions.
In this paper,  we further exploit this refined conjectural
information and show that it leads to better results on
these two celebrated problems of Linnik and Vinogradov.
 
For an accessible introduction to the  classical pair-correlation conjecture, 
see, e.g., Goldston \cite{PCintro}.

\subsection{Estimates on least quadratic non-residue}
For a given prime $q$, let $n(q)$ denote the least natural number that is not a quadratic
residue modulo $q$.  
The first breakthrough in obtaining a good upper bound for  $n(q)$ is
due to Vinogradov \cite{Vino1985}, 
who proved that
\(
n(q) \ll q^{\frac{1}{2\sqrt{e}}} (\log q)^2
\)
for all primes $q$, and \emph{conjectured} that
\begin{equation}\label{vino}
  n(q) \ll q^{\eps},
\end{equation}
for any fixed $\eps>0$.
This conjecture remains open unconditionally; the best
available bound (up to refinements of the $q^{\eps}$ factor) 
for general primes $q$ is
\(
n(q) \ll q^{\frac{1}{4\sqrt{e}}+ \eps},
\)
for any fixed $\eps > 0$. This is a well-known result of Burgess \cite{Burgess1957}.
  Linnik \cite{Linnik1942} demonstrated unconditionally that
the set of primes for which Vinogradov’s conjecture fails has density zero.

Next we consider conditional results.
In 1942, Linnik \cite{Linnik1942} showed that Vinogradov conjecture \eqref{vino} 
follows assuming GRH; in 1952, Ankeny \cite{ankeny} improved the bound further to
\(
n(q) \ll (\log q)^2
\)
under the same hypothesis.
Later in 2015, Lamzouri, Li and
Soundararajan \cite[Corollary 1.1]{LamzouriLS2015} refined this to
\(
n(q) \le (\log q)^2,
\)
for all $q  \ge  5$, again assuming GRH.
In the same year, Tao  \cite{Tao2015} 
proved that Vinogradov’s conjecture follows from the Elliott--Halberstam conjecture.
In 2016, assuming a conjectural improvement to bounds for long character sums, Bober and Goldmakher \cite{Bobgold} showed that 
$
n(q) \ll (\log q)^{1.4}
$
for primes $q \equiv 3 \pmod{4}$ and provided heuristic arguments suggesting that the exponent 1.4 should be replaceable by $1+\eps$.

Regarding lower bounds, assuming GRH, in 1971
Montgomery \cite[Theorem 13.15]{Montbook}  established that
\begin{equation*}
%\label{mont-nq}
n(q) = \Omega\bigl( (\log q) \log \log q\bigr).
\end{equation*}
In 1990, Graham and Ringrose \cite{GrahamR1990} \emph{unconditionally} proved that
\(
n(q) = \Omega\bigl( (\log q) \) \(\log \log \log q\bigr).
\) 
For a comprehensive overview of these developments, see \cite{MC2019}.

In this article we are interested in the implication of various pair-correlation hypotheses
for prime number theory. In particular, we show that a suitable pair-correlation 
hypothesis for the zeros of the quadratic Dirichlet $L$-function leads to a conditional improvement on the best known
upper bound for $n(q)$. This hypothesis is the direct analogue of Montgomery's pair-correlation conjecture for 
$\zeta(s)$. In order to state it, we start by defining
the pair-correlation functions for the zeros of Dirichlet $L$-functions associated with quadratic characters.

Throughout this note, $s$ denotes a complex number, $k$ a natural number, $q\ge 3$ a prime,
$\chi_{_\square}$ the quadratic character mod $k$ and $L(s, \chi_{_\square})$  the associated Dirichlet $L$-function. 
From now on, we assume the 
Riemann Hypothesis for Dirichlet $L$-functions associated with quadratic characters.
Let $\gamma_j$ be the imaginary part of the $j$-th zero (ordered by height on the half line)
and put
$$
F^+_{\chi_{_\square}}(x,T)
 \coloneq  
\sum_{\substack{0 \le  \gamma_j \le T,\ j = 1,2 \\ L(1/2+i\gamma_j, \chi_{_\square}) = 0}} 
x^{i(\gamma_1 - \gamma_2)} W(\gamma_1 - \gamma_2),
$$
where $W(u) =  4/(4+u^2)$.
We remark that, since $\chi_{_\square}$ is real, the zeros of the associated Dirichlet $L$ function 
are symmetrically distributed with respect to the real axis; hence $F^+_{\chi_{_\square}}(x,T) = F^-_{\chi_{_\square}}(x,T) $, where $$
F^-_{\chi_{_\square}}(x,T) 
 \coloneq  
\sum_{\substack{-T \le  \gamma_j \le 0,\ j = 1,2 \\ L(1/2 + i\gamma_j, \chi_{_\square}) = 0}} 
x^{i(\gamma_1 - \gamma_2)} W(\gamma_1 - \gamma_2).
$$
Our choice is to work with $F^+_{\chi_{_\square}}(x,T)$ only. Henceforth, whenever in this section we write $\gamma$ and $\gamma_j$ in the summation 
without additional specifications, we assume, that $L(1/2 + i\gamma, \chi_{_\square}) = 0$, respectively
$L(1/2 + i\gamma_j, \chi_{_\square}) = 0$.
Furthermore,  in  all of our sums, we count zeros of Dirichlet $L$-functions with \emph{multiplicity}.

Now, using classical results about the distribution of the zeros of the 
Dirichlet $L$-functions (\cite[Ch.\,16, eqn.\,1]{Davenport}  and  \cite[Ch.\,16, Lemma]{Davenport}), 
we deduce that

\begin{equation*}
%\label{trivial}
    |F^+_{\chi_{_\square}}(x,T)| 
    \le
    \sum_{\substack{0 \le  \gamma_j \le T ,\ j = 1, 2 %\\ L(1/2+i\gamma_j, \chi_j) = 0
    }}
\frac{4}{4 + (\gamma_1 - \gamma_2)^2}
\ll T (\log (kT))^2,
\end{equation*}
uniformly in $k,x$ as $T \to \infty$.

We first prove the following result for the range $x\le T \le e^x$, inspired by a
theorem  proved in 1973 by Montgomery \cite{Montgomery1973}.

\begin{Thm}\label{pc1chi}
Under RH for the Dirichlet $L$-function associated with $\chi_{_\square}$, the quadratic %primitive 
character
modulo $k$, asymptotically in $x$ we have
\[
F^+_{\chi_{_\square}}(x,T) 
\ll T \log (kx),
\]
uniformly for $3 \le k \le \exp\bigl(x (\log x)^{-1}\bigr)$
and $x \le T \le e^{x}$.
\end{Thm}

In the application we will only need upper bound estimates for $F^+_{\chi_{_\square}}(x,T)$.
For completeness, we will also give a result on its asymptotic behavior. 
This result requires GRH, since in a key point of the proof we use the 
prime number theorem in arithmetic progressions with error term.
\begin{Thm}\label{lam}
Under GRH, 
asymptotically in $x$ we have
\[
 F^+_{\chi_{_\square}}(x,T)  \sim \frac{T}{2\pi}\log x
 \]
uniformly for $ x \leq T \leq e^x$ and
$3\le k\le \sqrt{x}/(\log x)^{3}$.
\end{Thm}

\medskip 
%Theorem \ref{pc1chi} establishes the bound for $x\le T \le e^x$,
%the following hypothesis extends this behavior to the range $x^\eps \le T < x$.
%As mentioned before, it is this hypothesis that allows us
%to prove our result on $n(q)$ (see Theorem \ref{least-under-pc}).
Theorem \ref{pc1chi} gives an upper bound for $F^+_{\chi_{_\square}}(x,T)$  in
the range $x\le T \le e^x$,
the following hypothesis postulates that this bound remains valid in the 
larger range $x^\eps \le T < x$.

\begin{hypothesis}
\label{PCquadratic}
Under RH for the Dirichlet $L$-function associated with  $\chi_{_\square}$, the quadratic character
modulo $k$, asymptotically in $x$ we have
\[
F^+_{\chi_{_\square}}(x,T) 
\ll T \log (kx),
\]
uniformly for 
\[
3 \le k \le \exp\bigl(x (\log x)^{-1}\bigr)
 \quad\textrm{and} \quad
 x^{\eps} \le T < x.
  \]
\end{hypothesis}

This can be seen as the analogue of the pair-correlation conjecture used by Heath-Brown
\cite{HB1982} in 1982, but for the Dirichlet $L$-function associated with the
quadratic character modulo $k$.
In fact, in our application involving $n(q)$, we use the assumption only 
with $k=q$ a prime,
so from that perspective it suffices to state it only for prime moduli.

We are now ready to state the main result of this section. 
\begin{Thm}
\label{least-under-pc}
Let $\eps>0$,
$q$ be a sufficiently large prime and 
 $n(q)$ be the least quadratic non-residue modulo $q$.
 Assuming RH for Dirichlet $L$-functions associated with the quadratic
 character modulo $q$  and Hypothesis \ref{PCquadratic}, we have
\begin{equation}
\label{nq-estim}
n(q) \ll (\log q)^{1+\eps}.
\end{equation}
\end{Thm}

\begin{Rem}
If the $T$-uniformity range in Hypothesis \ref{PCquadratic} is enlarged to $(\log x)^2 \le T < x$, then 
the estimate \eqref{nq-estim} can be further refined to $n(q)\ll (\log q) (\log \log q)^2$.
\end{Rem}

\subsection{Linnik's constant}
Let $k$ and $a$ be coprime natural numbers with $1\le a\le k-1$. 
The least prime $p>a$ for which $p\equiv a\pmod{k}$ we denote by $p(k, a)$. The existence of $p(k,a)$ is guaranteed by  Dirichlet's classical theorem on primes in arithmetic progression, which states that there are infinitely many primes $p\equiv a\pmod{k}$.
A fundamental breakthrough was achieved by Linnik \cite{Linnik1944a,Linnik1944b}, who established that
\begin{equation}
\label{Linnik-estimate}
p(k, a) \ll k^{L}
\end{equation}
for some absolute constant $L$.
His proof is a masterpiece involving
a log-free zero-density theorem for Dirichlet $L$-functions
and a result on the Deuring--Heilbronn phenomenon, a ``repulsion'' effect due to 
the possible existence of the so-called ``exceptional'' zero that generates a larger
zero-free region for the Dirichlet $L$-functions.

Subsequent research has focused on refining the value of the exponent in \eqref{Linnik-estimate} (now known as ``Linnik's constant''). 
In 1992, Heath-Brown \cite{HB1992} proved
that $L \le 5.5$. In 2018, Xylouris \cite{Xylouris2011, Xylouris2018}, using improvement ideas indicated in
Heath-Brown's paper, lowered the bound to $L \le 5$, 
which has not been improved since.
 In the special case when the modulus $k$ is prime, Meng \cite{Men00} achieved $L=4.5$.

Let 
\begin{equation*}
%\label{pkdefinition}
p(k) \coloneq  \max_{\substack{1\le a \le k-1\\ (a,k)=1}} p(k,a).
\end{equation*}
In 1930, Titchmarsh \cite[Theorem 6]{Titchmarsh1930} proved, assuming 
GRH, that
\(
p(k) \ll k^2 (\log k)^4.
\)
In 1977, Kumar Murty \cite{kumarbscthesis} applied the Bombieri--Vinogradov theorem to prove that,
for every $\eps > 0$, one has
\(
p(k) < k^{2 + \eps}
\)
for almost all integers $k$. 
Assuming the Elliott--Halberstam conjecture, this bound may be improved to
\begin{equation}
\label{EH-conseq}
p(k) < k^{1 + \eps}
\end{equation}
for almost all $k$.
In \cite{mathannalen} we proved, assuming GRH and a suitable
form of the pair-correlation conjecture for Dirichlet $L$-functions,
Montgomery's conjecture on primes 
(as modified by Friedlander and Granville \cite[Conjecture 1(b), p.\,366]{FriedlanderG1989}), 
namely
\begin{equation}
\label{mont-conj-primes}
\psi(x;k,a) - \frac{x}{\varphi(k)} \ll \sqrt{\frac{x}{\varphi(k)}}\,x^\eps\ \text{~for~every~}\ 1\le k\le x^{1-\eps},
\end{equation}
holds true, where $\varphi$  denotes Euler's totient function.
As a consequence, \eqref{EH-conseq} follows for every $k$.
In 1978, Heath-Brown \cite{HB78} conjectured that
\(
p(k) \ll k (\log k)^2.
\)

In 1996,
Bach and Sorenson \cite{BachS1996}, assuming GRH, established that
\(
p(k,a) \le (1 + o(1)) (\varphi(k) \log k)^2,
\)
and also derived the explicit bound
$
p(k,a) \le 2 (k \log k)^2.
$
In 2015, Lamzouri, Li and Soundararajan
\,\cite[Corollary 1.2]{LamzouriLS2015} (see also \cite{LamzouriLS2017}) sharpened this, assuming GRH,  to
\(
p(k,a) \le (\varphi(k) \log k)^2,
\)
for $k>3$ and observed that one may derive an asymptotic bound of the form
\(
p(k,a) \le (1 - \delta + o(1)) (\varphi(k) \log k)^2
\)
for some small $\delta > 0$.

Our goal here is to obtain a (conditional) improvement of the upper bound for $p(k)$ given in \eqref{EH-conseq}.
By reworking our argument in \cite{mathannalen}, we will show
 that a suitable pair-correlation 
hypothesis for the zeros of Dirichlet $L$-functions will let us improve on \eqref{EH-conseq}. 

We define the pair-correlation functions for the zeros 
of Dirichlet $L$-functions analogous to those introduced in the previous section (and already used in our
earlier paper \cite{mathannalen}).
Assume the 
Riemann Hypothesis for Dirichlet $L$-functions (GRH).
Let $k$ be a natural number,
$\chi$ a Dirichlet character modulo $k$ and $L(s, \chi)$  the associated Dirichlet $L$-function. 
Let $\gamma_j$ be the imaginary part of the $j$-th zero (ordered by height on the half line).
Given two Dirichlet characters $\chi_1$ and $\chi_2$ modulo $k$, one can wonder to what 
extent the zeros of $L(s,\chi_1)$ are correlated with those of $L(s,\chi_2)$.
In order to measure this, we define
\begin{equation*}
%\label{Fqplus-def}
F^{+}_{k}(x,T) 
=
F^{+}_k(x,T;a)
\coloneq
\sum_{\chi_1,\chi_2 \pmod{k}} \overline{\chi_1}(a) \chi_2(a)G^{+}_{\chi_1, \chi_2} (x,T),
\end{equation*}
and
\begin{equation*}
%\label{Fq-def}
F_{k}(x,T) 
=
F_k(x,T;a)
\coloneq
\sum_{\chi_1,\chi_2 \pmod{k}} \overline{\chi_1}(a) \chi_2(a)G_{\chi_1, \chi_2} (x,T),
\end{equation*}
where
\[
G^{+}_{\chi_1,\chi_2}(x,T) 
\coloneq
\sum_{\substack{ 0 \le \gamma_j \le T, \, j = 1,2 \\ L(1/2+ i\gamma_j, \chi_j) = 0}} 
x^{i(\gamma_1 - \gamma_2)} W(\gamma_1 - \gamma_2),
\]
and
\[
G_{\chi_1,\chi_2}(x,T) 
\coloneq
\sum_{\substack{ \vert \gamma_j \vert \le T, \, j = 1,2 \\ L(1/2+ i\gamma_j, \chi_j) = 0}} 
x^{i(\gamma_1 - \gamma_2)} W(\gamma_1 - \gamma_2).
\]

Now, using classical results about the 
distribution of zeros of Dirichlet $L$-functions (\cite[Ch.~16, eq.~(1)]{Davenport} 
and  \cite[Ch.~16, Lemma]{Davenport}), 
we have 
\begin{equation*}
    \vert G^+_{\chi_1, \chi_2}(x,T) \vert  
   \le
   \sum_{\substack{0 \le  \gamma_j  \le T, \, j = 1, 2 \\ L(1/2+ i\gamma_j, \chi_j) = 0
   }}
\frac{4}{4 + (\gamma_1 - \gamma_2)^2}
\ll 
T (\log(kT))^2,
\end{equation*}
uniformly in $x$ as $T \to \infty$.
Therefore, under GRH we  have
\begin{equation*}
%\label{trivial1}
  F^{+}_{k}(x,T)
\ll T (\varphi(k) \log (kT))^2,
\end{equation*}
uniformly in $x$ as $T \to \infty$.

We recall the following result from \cite{mathannalen}, which provides an estimate for the pair-correlation function when $T$ is large relative to $x$:

\begin{Thm*}\cite[Theorem~2]{mathannalen}
\label{yalmodifiedcor}
Let  $\eps>0$ be arbitrary and fixed.
Under GRH, as \( x \to \infty \),  we have
\[
F_{k}(x,T)  
\ll \varphi(k)T \log x
\quad
\textrm{and}
\quad
F^{+}_{k}(x,T) 
\ll \varphi(k)T \log x
\]
uniformly for  every $a$ coprime with $k$ and
\[
1 \le k \le  x^{1-\eps}
\quad\textrm{and} \quad
x \le T
\le \exp(\sqrt{x}).
\]
\end{Thm*}

A discussion about the necessity of having to formulate separate results and conjectures  for both
$F_{k}(x,T)$ and $F^{+}_{k}(x,T)$ can be found  in the Introduction of \cite{mathannalen}.
In the same paper it was also explained why
the reduction of the exponent of $\varphi(k)$ from $2$ to $1$ 
in the estimates for $F_{k}(x,T)$ and $F^{+}_{k}(x,T)$ is the most important
feature of such a result. 
This reduction was achieved by exploiting the orthogonality of Dirichlet
characters in the definition of $F_{k}(x,T)$ and $F^{+}_{k}(x,T)$.

In \cite[Remark 1]{mathannalen}  we also noted that the $k$-uniformity range
in Theorem \ref{yalmodifiedcor} can be extended to $1 \le k \le  x/ (\log x)^{1+\eps}$ 
at the cost of weaker estimates, with the bounds for $F_{k}(x,T)$ and 
$F_k^{+}(x,T)$ being of order $\varphi(k)\, T (\log x)^{2}$.

We will show that having estimates for 
$F_{k}(x,T)$  and $F^{+}_{k}(x,T)$ in the range
where $T$ is small relative to $x$, allows one to obtain  estimates for primes in arithmetic progressions sharper than
\eqref{mont-conj-primes}. Specifically, we will consider the implications of the following hypothesis.

\begin{hypothesis}\label{conjgpc2}
Let $c_1>0$ and $c_2\in (1/2,1)$ be fixed.
Under GRH, we have, as \( x \to \infty \), 
\[
F_{k}(x,T)  
\ll \varphi(k)T 
\exp\bigl(c_1(\log x)^{c_2}\bigr)
\quad
\textrm{and}
\quad
F^{+}_{k}(x,T)
\ll \varphi(k)T 
\exp\bigl(c_1(\log x)^{c_2}\bigr)
\]
uniformly for 
\[
\exp\bigl(c_1(\log x)^{c_2}\bigr)
<   ~ k \le 
x\exp\bigl(-2c_1(\log x)^{c_2}\bigr)
\quad\textrm{and} \quad
\exp\bigl(c_1(\log x)^{c_2}\bigr)
\le T
< x.
\]
\end{hypothesis}

Using it, in conjunction with the GRH, we are able to prove the following result.
 
\begin{Thm}\label{mor5} 
Let $c_1>0$ and $c_2\in (1/2,1)$ be fixed.
Assume that both GRH and Hypothesis \ref{conjgpc2} 
hold true.  
Then
\begin{equation}
\notag
 \psi(x;k,a) - \frac{x}{\varphi(k)} \ll \sqrt{\frac{x}{\varphi(k)}}\,\exp\bigl(c_1(\log x)^{c_2}\bigr)
\end{equation}
for $x$ sufficiently large and
uniformly for
$ 1 \le  k \le x \exp\bigl(-2c_1(\log x)^{c_2}\bigr)$ and $1\le a\le k$ with $(a,k)=1$.
\end{Thm}
The quality of the upper bound and of the $k$-uniformity level in Theorem \ref{mor5} 
are crucial features.
In an email communication about a previous version of this work,
Granville remarked that the results in \cite{FriedlanderG1989} imply that
\begin{equation}
\label{FG_cond1}
 \psi(x;k,a) - \frac{x}{\varphi(k)} \gg u^{-u}\frac{x}{\varphi(k)},
\end{equation}
where $u$ is such that $x/k = (\log x)^u$ with $1\le k \le x$. 
As a consequence, 
Theorem \ref{mor5} 
with $\exp\bigl(c_1(\log x)^{c_2}\bigr)$ replaced by
$(\log x)^A$ uniformly for $ 1 \le  k \le x (\log x)^{-B}$, $A,B>0$,
cannot hold true.
Moreover, Conjecture 1(a) of \cite[p.~366]{FriedlanderG1989} states that 
\begin{equation}
\label{FG_cond2}
 \psi(x;k,a) \ll_\eps \frac{x}{\varphi(k)} 
\end{equation}
is admissible for  $1\le k \le x/(\log x)^{2+\eps}$. 
We note that the $k$-uniformity level 
in the statement of Hypothesis \ref{conjgpc2} is sharp, in the sense that otherwise the resulting variant
of Theorem \ref{mor5} would be in contradiction with formulae 
\eqref{FG_cond1}-\eqref{FG_cond2}.

We also remark that the sub-range $1\le k \le \exp\bigl(c_1(\log x)^{c_2}\bigr)$ in Theorem \ref{mor5}
is handled by using the well-known estimate
\begin{equation}
\label{GRH-only}
\psi(x;k,a)- \frac{x}{\varphi(k)}
\ll \sqrt{x} (\log x)^2,
\end{equation}
that can be established using GRH.

The technique we use in proving Theorem \ref{mor5} 
is similar to that used by Heath-Brown in his proof of Theorem 1 in \cite{HB1982}.

As a consequence of Theorem \ref{mor5} we obtain the following estimate for $p(k)$.
\begin{Thm}
\label{least-prime-gpc}
Given coprime natural numbers $a$ and $k$ with $1\le a\le k-1$, denote by
$p(k, a)$ be the least prime  
greater than $a$ such that $p \equiv a \pmod{k}$.
Assume that both GRH and Hypothesis \ref{conjgpc2}
hold true.  Then 
there exist $A\in (1/2,1)$ and $B>0$ such that
\[
p(k)=\max_{\substack{1\le a \le k-1 \\ (a,k)= 1}} p(k,a)\ll \varphi(k)\exp\bigl(B(\log k)^{A}\bigr).
\]
\end{Thm}
The constants $A,B$ appearing above correspond with the constants $c_2$, respectively $2c_1$
in Theorem \ref{mor5}.

\begin{Rem}
Unfortunately, this technique does not bring us closer
to establishing Heath-Brown's conjectural estimate $p(k)\ll k (\log k)^2$.
Namely, we would need a larger uniformity range for $k$ in Hypothesis \ref{conjgpc2}
and in Theorem \ref{mor5}, 
contradicting the results of Friedlander and Granville
\cite{FriedlanderG1989} just described above.
\end{Rem}

\section{Preliminaries}\label{pre}
\subsection{Prime number distribution}
In this section, we recall the material we need on the distribution of prime numbers, using the notations
\begin{equation*}
\pi(t) \coloneq \sum_{p \le  t} 1, \quad \quad
\pi(t;k,a) \coloneq \sum_{\substack{p \le  t \\ p \equiv a  \pmod{k}}} 1,
\end{equation*}
and
\begin{equation*}
\psi(t) \coloneq \sum_{n \le  t} \Lambda(n), \quad \quad
\psi(t;k,a) \coloneq \sum_{\substack{n \le  t \\ n \equiv a  \pmod{k}}} \Lambda(n),  
\quad \quad 
\psi(t,\chi) \coloneq \sum_{n \le t} \Lambda(n) \chi(n),
\end{equation*}
where $\Lambda$ denotes the von Mangoldt function and $\chi$ a Dirichlet character
modulo $k$.
Two equivalent forms of the Prime Number Theorem are
\begin{equation*}
%\label{pnt}
\pi(t)\sim
\int_{2}^{t} \frac{\, \dx u}{\log u}
\quad
\textrm{and}
\quad
\psi(t)\sim t
\quad (t\to \infty).
\end{equation*}
For fixed coprime integers $a$ and $k$, we have asymptotic equidistribution: 
\begin{equation*}
%\label{equi}
\pi(t;k,a)\sim \frac{\pi(t)}{\varphi(k)},
\quad
\textrm{and}
\quad
\psi(t;k,a)\sim \frac{\psi(t)}{\varphi(k)}
\quad (t\to \infty).
\end{equation*}

An important tool we will use is the following theorem (for a proof, see, e.g.,
Montgomery--Vaughan \cite[Sect.~3--6]{MVsieve}).
\begin{CThm}[Brun--Titchmarsh theorem]
\label{BT-thm}
Let $x,y>0$ and $a,k$ be coprime positive integers.
Then, uniformly for all $y>k$, we have
\begin{equation*} 
%\label{BT-estim} 
\pi(x+y;k,a) - \pi(x;k,a) < \frac{2y}{\varphi(k) \log(y/k)}.
\end{equation*}
\end{CThm}

Starting point for our deliberations is an explicit truncated form of the 
von Mangoldt explicit formula, which we state in the classical, respectively 
Dirichlet $L$-function case.

\begin{lem}\cite[Ch.~17]{Davenport}.
%\label{davenlem2}
Let $2 \le Z \le x$. Assuming RH, we 
have, as $ x \to \infty,$
\[
\psi(x) = x - \sum_{\substack{ \vert \gamma \vert  \le  Z \\
\zeta(1/2 + i\gamma) = 0}} \frac{x^{1/2 + i\gamma}}{1/2 + i\gamma} +
O\Bigl(\frac{x}{Z} (\log (xZ))^2 \Bigr).
\]
\end{lem}

\begin{lem}\cite[Ch.~19]{Davenport}.\label{davenlem}
If $\chi$ is a non principal 
character modulo $k$ and $2 \le Z  \le $ $x$, then assuming RH for $L(s,\chi)$ holds, we have, as $ x \to \infty,$ that
\begin{equation}
\label{explform-all-gammas}
\psi(x,\chi) = - \sum_{\substack{ \vert \gamma \vert  \le  Z \\ L(1/2+i\gamma)=0
}} \frac{x^{1/2 + i\gamma}}{1/2 + i\gamma} +
O\Bigl(\frac{x}{Z} (\log (kxZ))^2 \Bigr).
\end{equation}
\end{lem}
\noindent
As an immediate consequence we have:
\begin{lem} \cite[Ch.~19--20]{Davenport}.
\label{davenlem3}
If $\chi$ is a non principal 
character modulo $k$, then assuming RH for $L(s,\chi)$ holds, we have, as $ x \to \infty,$ that
\begin{equation*}
\psi(x,\chi) \ll \sqrt{x} (\log (kx))^2.
\end{equation*}
\end{lem}

We will apply the Cauchy--Schwarz inequality ($L^2$-norm form) multiple times
in our proofs and so we state it here: 
\begin{equation}
\label{Cauchy}
\Bigl \vert  \int_a^b f(t) g(t) \, \dx t \Bigr \vert  
\le 
\Bigl( \int_a^b  \vert f(t) \vert ^2  \, \dx t \Bigr)^{1/2} 
\Bigl( \int_a^b  \vert g(t) \vert ^2 \, \dx t \Bigr)^{1/2},
\end{equation}
where $f$ and $g$ are arbitrary square-integrable complex functions.

We also record the Sobolev--Gallagher inequality (which will play an important role in proving Lemma \ref{lem3}).
\begin{lem} \cite[Lemma~1.1]{Montbook}
\label{Sob-Gal-ineq}
Let $a<b$ be real numbers and $f$ a continuous complex valued function on
$[a,b]$, with continuous first derivative $f'$ in $(a,b)$. Then
\begin{equation*}
%\label{SG}
\vert f(u) \vert  \le \frac{1}{b-a} \int_a^b  \vert f(t) \vert  \, \dx t + \int_a^b  \vert f'(t) \vert  \, \dx t
\end{equation*}
for any $u$ in $[a, b]$.
\end{lem}

Finally, we state
an analogue of Lemma 6 of Goldston--Montgomery
\cite{GM1984}. It will be crucially used in the proof of Lemma \ref{GM-lemma}.
\begin{lem}\cite[Lemma~6]{GM1984}
\label{gmslem}
Let $\mathcal{S}(t) \coloneq \sum_{\mu \in \mathcal{M}} c(\mu) e^{2\pi i\mu t}$
be a Fourier series with $\mathcal{M}$ be a countable set of real numbers and with 
$c(\mu)$ real Fourier coefficients. 
If $\sum_{\mu \in \mathcal{M}}  \vert c(\mu) \vert  < \infty$,
then uniformly for $T\ge 1$ and $1/T\le \delta \le 1$, we have

\[
\int_{0}^{T}  \vert \mathcal{S}(t) \vert ^2 \, \dx t = 
T \sum_{\mu \in \mathcal{M}}  \vert c(\mu) \vert ^2 + 
O\Big( \delta^{-1} \sum_{\mu \in \mathcal{M}}  \vert c(\mu) \vert ^2 
+ 
T\!\! \sum_{\substack{\mu, \nu \in \mathcal{M} \\
0 <  \vert \mu-\nu \vert < \delta}}  \vert c(\mu) c(\nu) \vert \Big).
\]
\end{lem}

\section{Proofs of Theorems \ref{pc1chi} and \ref{lam}}
First of all we remark that if
 $\chi'_{_{\square}} \pmod{k_1}$ is the primitive character that induces $\chi_{_{\square}}
\pmod{k}$ then $F^+_{\chi'_{_\square}}(x,T) = F^+_{\chi_{_\square}}(x,T) $ and $k_1 | k$.
Therefore, it is sufficient to prove Theorem \ref{pc1chi} for primitive quadratic characters only.

Under RH for the Dirichlet $L$-function associated with the quadratic character
modulo $k$, following the argument in Landau \cite[p.~353]{LandauHand} we obtain for $k\ge 3$,
$(a,k)=1$, $x > 1$, and for primitive quadratic characters $\chi_{_{\square}}$, the identity
\[
\sideset{}{'} \sum_{n\le  x} \frac{\Lambda(n)\chi_{_{\square}}(n)}{n^s}
=
\sum_{\ell=0}^{\infty}\frac{x^{-2\ell-s-\mathfrak{a}}}{2\ell+s+\mathfrak{a}}
-
\sum_{\rho}\frac{x^{\rho-s}}{\rho-s}
-
\frac{L^\prime}{L}(s,\chi_{_{\square}}),
\]
where the dash over the summation symbol indicates that only half of
the term with 
$n = x$ is to be included in the sum. As is customary,
the sum over the non-trivial zeroes has to be interpreted in 
the symmetrical 
sense as $ \lim_{Z \to \infty} \sum_{ \vert \gamma \vert  < Z}$,
$s\in \mathbb C$,
%$s\ne 1$,
$s\ne \rho$, $s\ne -(2\ell+\mathfrak{a})$,
$\mathfrak{a} =0$ if $\chi_{_{\square}}(-1)=1$, and
$\mathfrak{a} =1$ if $\chi_{_{\square}}(-1)=-1$.
Letting $s=\sigma+it$ and $\rho=1/2+i\gamma$, 
and following the proof of Montgomery \cite[pp.~185--186]{Montgomery1973} (see also, \cite[\S 3]{mathannalen}),
the previous formula can be rewritten as
\begin{equation}
\label{switchLR}
\sum_{\rho}
\frac{x^{1/2-\sigma}x^{i(\gamma-t)}}{\sigma-1/2 +i(t-\gamma)}
=
%\Bigl(
 \frac{L^\prime}{L}(s,\chi_{_{\square}})
 +
\sideset{}{'} \sum_{n\le  x} \frac{\Lambda(n)\chi_{_{\square}}(n)}{n^s}
-
\sum_{\ell=0}^{\infty}\frac{x^{-2\ell-s-\mathfrak{a}}}{2\ell+s+\mathfrak{a}}
%\Bigr)
.
\end{equation}
On replacing $s$ by $1-\sigma+it$, we obtain
\begin{align}
\notag
%\label{reflection}
\sum_{\rho}\frac{x^{\sigma-1/2}x^{i(\gamma-t)}}{1/2-\sigma +i(t-\gamma)}
=
\frac{L^\prime}{L}(1-\sigma+it,\chi_{_{\square}})
& +
\sideset{}{'} \sum_{n\le  x} \frac{\Lambda(n)\chi_{_{\square}}(n)}{n^{1-\sigma+it}}
\\& 
\label{reflection}
-
\sum_{\ell=0}^{\infty}\frac{x^{-2\ell-1 + \sigma -it-\mathfrak{a}}}{2\ell+1 - \sigma +it+\mathfrak{a}}
.
\end{align}
Subtracting the respective sides of \eqref{reflection} from \eqref{switchLR} 
and using the relation 
\begin{equation}
\label{derlog-L}
 \frac{L^\prime}{L}(s,\chi_{_{\square}})
 =
 - \sum_{n\ge 1} \frac{\Lambda(n)\chi_{_{\square}}(n)}{n^s}
 \quad (\Re(s)>1),
\end{equation}
 we obtain 
\begin{align}
\notag
\sum_{\gamma}
\frac{(2\sigma-1) x^{i\gamma}}{(\sigma -1/2)^{2}+(t-\gamma)^{2}}
=
& -\frac{1}{\sqrt{x}} 
\Bigl(
\sideset{}{'} \sum_{n\le x} 
\Lambda(n)\chi_{_{\square}}(n) \Bigl( \frac{x}{n} \Bigr)^{1-\sigma+it} 
\\&
\notag
+
\sum_{n > x} \Lambda(n)\chi_{_{\square}}(n) \Bigl( \frac{x}{n} \Bigr)^{\sigma+it} 
\Bigr)
 \\
&
%\hskip2.5cm
-
\frac{L^\prime}{L}(1-\sigma+it, \chi_{_{\square}})
x^{1/2-\sigma+it} \nonumber \\
\label{Montgomery-eq-(22)}
&
-\frac{1}{\sqrt{x}} 
\sum_{\ell=0}^{\infty}
\frac{(2\sigma-1) x^{-2\ell-\mathfrak{a}}}{(\sigma-1-it-2\ell-\mathfrak{a} ) (\sigma+it+2\ell+\mathfrak{a})}.
\end{align}
For primitive $\chi_{_{\square}} \pmod{k}$
we obtain, by logarithmic differentiation of both sides appearing in the functional equation of 
$L(s,\chi_{_{\square}})$, the estimate

\begin{equation*}
%\label{derlog-reflection}
- \frac{L^\prime}{L}(1-\sigma+it, \chi_{_{\square}})
=
\log \Bigl(\frac{k\tau}{2\pi} \Bigr)
+ \frac{L^\prime}{L}(\sigma-it, \chi_{_{\square}})
+O_\sigma (1),
\end{equation*}
with \( \tau =  \vert t \vert +2\), 
which,  using \eqref{derlog-L}, for $\sigma>1$ becomes
 \begin{equation}
\label{derlog-reflection}
- \frac{L^\prime}{L}(1-\sigma+it, \chi_{_{\square}})
=
\log (k\tau)
+O_\sigma (1).
\end{equation}
Letting $\sigma=3/2$, inserting \eqref{derlog-reflection} into
\eqref{Montgomery-eq-(22)}, we have
\begin{align}\label{eqn8}
\Bigl \vert 
\sum_{\substack
{\gamma
%\colon L(\frac{1}{2} + i\gamma, \chi_{_{\square}}) = 0
}} \frac{2x^{i\gamma}}{1+(t-\gamma)^2} \Bigr \vert ^2
= \Bigl \vert \sum_{1\le  j \le  3} R_j(x,t)\Bigr \vert ^2,
\end{align}
where 
\begin{align*}
R_1(x,t) & \coloneq 
- \frac{1}{\sqrt{x}} 
   \Bigl( 
   \sideset{}{'}\sum_{\substack{n \le x }} \Lambda(n) \chi_{_{\square}}(n) \Bigl(\frac{x}{n}\Bigr)^{-1/2+it}
   + 
   \sum_{\substack{n > x }}\Lambda(n) \chi_{_{\square}}(n)
   \Bigl(\frac{x}{n}\Bigr)^{3/2+ it}
   \Bigr),\\ 
 R_2(x,t) &\coloneq x^{-1+it} 
  ( \log (k \tau) +O(1)), 
\\
R_3(x,t) & \coloneq O( x^{-1/2}\tau^{-1}).
\end{align*} 

The precise definition of $R_2$ and $R_3$ is not relevant for our purposes, the estimates above will suffice. 

\begin{proof}[Proof of Theorem \ref{pc1chi}]

We integrate both sides of equation \eqref{eqn8} from $t=0$ to $t=T$,
where $T$ will be specified later.
The left hand side of equation \eqref{eqn8} can be written as 
\begin{equation*}
    \sum_{\substack{\gamma_j, \, j=1,2 
    %\\ L(\frac{1}{2} + i\gamma_j, \chi_{_{\square}}) = 0
    }}
   \frac{4 x^{i(\gamma_1 - \gamma_2)}}{(1+(t -\gamma_1)^2) (1+(t-\gamma_2)^2)},
\end{equation*}
Now we integrate 
this function from $t=0$ to $t=T.$ We claim that
\begin{align}
\notag
    \int_{0}^T
    \sum_{\substack{ \gamma_j, \, j=1,2
    }}&
   \frac{4 x^{i(\gamma_1 - \gamma_2)}}{(1+(t -\gamma_1)^2) (1+(t-\gamma_2)^2)} \, \dx t 
   \\
\notag
    & =
    \int_{- \infty}^{\infty} 
    \sum_{\substack{  0 \le \gamma_j  \le T \\ j=1,2
     }}
   \frac{4 x^{i(\gamma_1 - \gamma_2)}}{(1+(t -\gamma_1)^2) (1+(t-\gamma_2)^2)} \, \dx t 
   + O\bigl( (\log T)(\log (kT))^2\bigr)
   \\
     \label{sona}
     &=   2\pi 
     F^+_{\chi_{_\square}}(x,T) + O\bigl( (\log T)(\log (kT))^2\bigr).
   \end{align}
To prove equation \eqref{sona}, we first recall that for $\chi_{_{\square}} \pmod{k}$, there are $\ll \log (kT)$ 
zeros such that $L(1/2 + i\gamma, \chi_{_{\square}}) = 0$ and $T \le \gamma \le T+1$, $T  \ge  2$. 
This implies for $ 0 \le t  \le T$ that
\begin{equation}
\label{eqn33}
    \sum_{\substack{ \gamma   >T 
    %\\ L(\frac{1}{2} + i\gamma, \chi_{_{\square}}) = 0
    }}
   \frac{1}{1+(t - \gamma)^2 }
   \ll 
    \frac{\log (kT)}{T-t+1}.
    \end{equation}

We now recall \cite[Ch.~16, Lemma]{Davenport}, i.e., 

\begin{equation}
\label{eqn34}
    \sum_{\substack{\gamma 
    %\\ L(\frac{1}{2} + i\gamma, \chi_{_{\square}}) = 0
    }}
   \frac{1}{1+(t - \gamma)^2 }
   \ll \log (k \tau).
    \end{equation}
Using \eqref{eqn33}-\eqref{eqn34} we obtain  that
\begin{equation}
    \int_{ 0}^T
   % &
    \sum_{\substack{\gamma_1, \gamma_2 \\  \gamma_2   >T 
    %\\ L(\frac{1}{2} + i\gamma_j, \chi_{_{\square}}) = 0
    }}
   \frac{4 x^{i(\gamma_1 - \gamma_2)}}{(1+(t -\gamma_1)^2) (1+(t-\gamma_2)^2)} \, \dx t
   %\\&
   \label{eqn36}
   \ll  (\log T)(\log (kT))^2.
\end{equation}
For $ t  > T$, we have 
\begin{align*}
\sum_{ 0 \le \gamma \le T}
\frac{1}{1+(t-\gamma)^{2}}
&\ll
\frac{\log (k t)}{ t - T + 1},
\end{align*}
so that
\begin{align}
    \int_{  t > T}  \sum_{\substack{ 0 \le  \gamma_j  \le T \\ j=1,2 
    %\\ L(\frac{1}{2} + i\gamma_j, \chi_{_{\square}}) = 0
    }}
   \frac{\, \dx t}{(1+(t -\gamma_1)^2) (1+(t-\gamma_2)^2)}  \nonumber 
   & \ll 
   \int_{  t  > T}  
   \frac{\log^2 (k t) }{(t  - T + 1 )^2} \, \dx t
   %\\ \nonumber &
  \ll (\log (kT))^2. 
\end{align}
This implies  that
\begin{equation}
%\notag
    \int_{  t > T}  
   % &
    \sum_{\substack{ 0 \le \gamma_j   \le T \\ j=1,2 %\\ L(\frac{1}{2} + i\gamma_j, \chi_{_{\square}}) = 0
    }}
   \frac{4 x^{i(\gamma_1 - \gamma_2)}}{(1+(t -\gamma_1)^2) (1+(t-\gamma_2)^2)} \, \dx t 
   %\\&
   \label{eqn35}
   \ll  (\log (kT))^2.
\end{equation}
Now \eqref{sona} follows on 
combining equations \eqref{eqn36}-\eqref{eqn35}
and applying Cauchy's 
residue theorem to evaluate the second integral in \eqref{sona}.
From \eqref{eqn8}-\eqref{sona} we obtain
   \begin{align}\label{intlhsrhs}
       \int_{0}^T \big \vert \sum_{1\le  j \le  3} & R_j(x,t)\big \vert ^2 \, \dx t 
  =  2 \pi F^+_{\chi_{_\square}}(x,T) + O\bigl( (\log T)(\log (kT))^2\bigr).
   \end{align}
Using the Cauchy--Schwarz inequality \eqref{Cauchy}, we can easily deduce that
\begin{align}\label{eqn39}
  \int_{0}^T \big \vert \sum_{1\le  j \le  3} & R_j(x,t)\big \vert ^2 \, \dx t 
  =   \sum_{1\le  j\le  3} \int_{0}^T \big \vert  R_j(x,t)\big \vert ^2 \, \dx t   \nonumber \\
  & + O\Bigl(\sum_{1\le  j \le  3} \sum_{\substack{1\le  \ell \le  3\\ \ell\ne j}} 
  \bigl(\int_{0}^T \big \vert  R_j(x,t)\big \vert ^2 \, \dx t\bigr)^{1/2} 
  \bigl(\int_{0}^T \big \vert  R_\ell(x,t)\big \vert ^2 \, \dx t\bigr)^{1/2}
  \Bigr).
\end{align}

Integrating   $\vert R_j(x,t) \vert^2, j = 2,3$ (see \eqref{eqn8} for 
their definitions) from $t=0$ to $t=T$, we 
obtain, for $k \ge  3$, that

\begin{equation}\label{eqn411}
    \int_{0}^T  \vert R_2(x,t) \vert ^2 \, \dx t \ll
    \frac{ T }{x^2} (\log (kT))^2,
\end{equation}
and 
\begin{equation}\label{eqn40}
   \int_{0}^T  \vert R_3(x,t) \vert ^2 \, \dx t  \ll 1.
\end{equation}

The mean square of  $R_1(x,t)$ is evaluated in the following lemma.
\begin{lem}
\label{GM-lemma}
For $ x \le T$ we have, as $x \to \infty$,
\begin{align}
\label{sx11}
   \int_{0}^T  \vert R_1(x,t) \vert ^2 \, \dx t  &=
  T S(x) + O\bigl(\sqrt{\, TxS(x)}\bigr),
   \end{align}
where
\begin{equation}
\label{Sdef}
 S(x) \coloneq \frac{1}{x^2}
\sum_{\substack{n \le x \\ (n,k)=1}} n \Lambda(n)^2 +
x^2 \sum_{\substack{n > x \\ (n,k)=1 }}\frac{\Lambda(n)^2}{n^3}.
\end{equation}
\end{lem}

\begin{proof}
Recalling the definition of $R_1(x,t)$ given in \eqref{eqn8}, we can write
\begin{equation}
\label{eqn388ll}
  - R_1(x,t) = \sum_{n  \ge  1} c(n) \Bigl(\frac{x}{n}\Bigr)^{it}, 
   \end{equation}
where 
$c(n) \coloneq	 \Lambda(n) \chi_{_{\square}}(n) n^{1/2}x^{-1}$, if  $n\le x$
and $c(n) \coloneq \Lambda(n) \chi_{_{\square}}(n) n^{-3/2} x$, if $ n> x$.
Using the prime number theorem it is clear that $\sum_{n} |c(n)|^2 < \infty.$
Applying Lemma \ref{gmslem} to the series in \eqref{eqn388ll}, 
we obtain 
\begin{align}\label{r1ts}
     \int_{0}^T  \vert R_1(x,t) \vert ^2 \, \dx t = T S(x) +
  O\bigl(\delta^{-1} S(x) + E_1 + E_2\bigr),
\end{align}
where
\[
E_1 \coloneq T \sum_{ \substack{n < x \\ (n,k)=1}} \Big( \sum_{\substack{ m \le x
\\ (m,k)=1 \\ 0 <  \vert \log(n/m) \vert <2\pi \delta}}\!\!\!\!\! \Lambda(n) \Lambda(m) \frac{(nm)^{1/2}}{x^2}+
\!\!\!
\sum_{\substack{ m > x \\ (m,k)=1 \\ 0 <  \vert \log(n/m) \vert <2\pi \delta}}\!\!\!\!\! \Lambda(n) \Lambda(m) \frac{n^{1/2}}{m^{3/2}}
\Big),
\]
and

\[
E_2 \coloneq T
\sum_{\substack{n  \ge  x \\ (n,k)=1}}
\Big( \sum_{\substack{ m \le x \\ (m,k)=1 \\ 0 <  \vert \log(n/m) \vert <2\pi \delta}} 
\!\!\!\!\!\!\!\!\!\!\Lambda(n) \Lambda(m) \frac{m^{1/2}}{n^{3/2}} 
+
\!\!\!\!\!\!\!\!\!\!
\sum_{\substack{m > x \\ (m,k)=1  \\ 
0 <  \vert \log(n/m) \vert <2\pi \delta}} 
\!\!\!\!\!\!\!\!\!\Lambda(n) \Lambda(m) \frac{x^2}{(nm)^{3/2}}
\Big).
\]

We obtain, for all $k \ge  3$
\begin{equation}
\label{pntsx}
   S(x) \ll \log x,
\end{equation}
as $x \to \infty.$
Choosing 
\begin{equation}
\label{delta-def}
\delta = \sqrt{\frac{S(x)}{Tx}},
\end{equation}
for sufficiently large $x$, we have that $1/T \le  \delta \le 1,$
when $ x \le T$.
Since $ e^{2\pi \delta} - 1 \le  10^4 \delta$ for any $\delta \le 1$, we obtain

\begin{align}\label{E1est}
\nonumber
   E_1 &\ll
   \frac{T}{x} \sum_{ \substack{n \le  x}} \sum_{\substack{ m  \ge  1
\\ 0 <  \vert \log(n/m) \vert <2\pi \delta}}\!\!\!\!\! \Lambda(n) \Lambda(m)
%\\&\nonumber
    \ll 
    \frac{T}{x}
\sum_{\substack{n \le 
% 10^4
x } } \Lambda(n) 
\sum_{ \substack{n < m \le n+ 10^4 \delta x}}  \Lambda(m) 
   \\
   &\ll \frac{T}{x}
\sum_{\substack{n  \le 
%10^4
x } } \Lambda(n) (\psi(n + 10^4 \delta x) -  \psi(n) )
   \ll  T \delta x,
\end{align}
for sufficiently large $x$, where we have used the prime number theorem in the final step.  
Moreover, we can write
\begin{align}\label{E2est}
   E_2 &= T
\sum_{ r = 0}^{\infty} \sum_{ \substack{n= 2^r x}}^{2^{r+1}x}
\Big(\sum_{\substack{ m \le x \\ 0 <  \vert \log(n/m) \vert <2\pi \delta}} 
\!\!\!\!\!\!\!\!\!\!\Lambda(n) \Lambda(m) \frac{m^{1/2}}{n^{3/2}} 
+
\!\!\!\!\!\!
\sum_{\substack{m > x \\
0 <  \vert \log(n/m) \vert <2\pi \delta}} 
\!\!\!\!\!\!\!\!\!\Lambda(n) \Lambda(m) \frac{x^2}{(nm)^{3/2}}
\Big) \nonumber\\ 
& \ll 
   \frac{T}{x} \sum_{r=0}^{\infty} 
    \frac{1}{2^{3r}} \sum_{ \substack{n= 2^r x}}^{2^{r+1}x}
\sum_{\substack{ m  \ge  1 
\\ 0 <  \vert \log(n/m) \vert <2\pi \delta}} \Lambda(n) \Lambda(m)
\nonumber\\ 
& \ll 
   \frac{T}{x} \sum_{r=0}^{\infty} 
    \frac{1}{2^{3r}} 
    \sum_{\substack{n \le  2^{r+1} x}}
\sum_{\substack{ m  \ge  1 
\\ 0 <  \vert \log(n/m) \vert <2\pi \delta}} \Lambda(n) \Lambda(m)
\nonumber\\ 
& \ll  \sum_{r=0}^{\infty}  \frac{T \delta x}{2^{r}} 
   \ll T \delta x,
\end{align}
as $x \to \infty,$
 where the final single sum  is obtained on using \eqref{E1est}.
Combining  \eqref{r1ts} and \eqref{E1est}-\eqref{E2est}, and substituting the value of $\delta$ 
given in \eqref{delta-def}, we have completed the proof of Lemma \ref{GM-lemma}.
\end{proof}

{}From \eqref{sx11} and \eqref{pntsx} we deduce that
\begin{align}\label{eqn3899}
\int_{0}^T  \vert R_1(x,t) \vert ^2 \, \dx t 
  \ll  T \log x + \sqrt{\, Tx\log x},
\end{align}
as $x \to \infty$.

Combining \eqref{eqn39}-\eqref{eqn40} and \eqref{eqn3899} it follows that

\begin{align}\label{eqn42new}
  \int_{0}^T \big \vert \sum_{1\le  j \le  3} R_j(x,t)\big \vert ^2 \, \dx t & \ll
 T \log x + \sqrt{\,Tx\log x} +
  \frac{T}{x^2} (\log (kT))^2
  \ll  T \log (kx),
 \end{align}
\noindent
uniformly in the range $3\le k \le e^x$ and
$x \le T \le e^x$ as $x \to \infty.$
Recalling the estimate \eqref{intlhsrhs} and exploiting the monotonicity of $(\log T)^3/ T$  for $T$ sufficiently large, we also have
$ (\log T)(\log (kT))^2 \ll T \log (kx)$
in the restricted range $3 \leq k \le \exp(x (\log x)^{-1})$ and $x \le T \le e^x$,
as $x \to \infty$.
Summarizing, we have shown that \eqref{intlhsrhs} and \eqref{eqn42new} allow one to show that the estimate
\[
   F^+_{\chi_{_\square}}(x,T)  \ll  T  \log (kx),
\]
holds uniformly in the range $3 \leq k \le \exp(x (\log x)^{-1})$
and $x \le T \le e^x$  as $x \to \infty.$
\end{proof}

\begin{proof}[Proof of Theorem \ref{lam}]
 
For $(a,k)= 1$, let 
\[
S(x;k,a) \coloneq \frac{1}{x^2}
 \sum_{\substack{n \leq x \\
 n \equiv a \pmod{k}}} n \Lambda(n)^2 +
x^2 \sum_{\substack{n > x \\ \nonumber
 n \equiv a \pmod{k}}}\frac{\Lambda(n)^2}{n^3}.
\]
In \cite[eqn (35)]{mathannalen}, under the assumption of GRH we proved that
\[
S(x;k,a) = \frac{\log x}{\varphi(k)} + O \Big(\frac{(\log x)^3} {\sqrt{x}}  \Big),
\]
as $x \to \infty$, uniformly in $1\le k \leq \sqrt{x}/(\log x)^{3}$.
Hence, summing over residue classes, we obtain
that $S(x)$ defined in \eqref{Sdef} satisfies the asymptotic
\begin{equation}\label{sx}
    S(x)  = \log x + O\Bigl(\varphi(k) \frac{(\log x)^3}{\sqrt{x}}\Bigr)
    = (\log x)(1+o(1)),
\end{equation}
uniformly for $1\le k \le \sqrt{x}/(\log x)^3$.
From Lemma \ref{GM-lemma}, see \eqref{sx11}, and \eqref{sx} it follows that
\begin{align}\label{eqn42neww}
  \int_{0}^T \big \vert \sum_{1\le  j \le  3} R_j(x,t)\big \vert ^2 \, \dx t & =
T (\log x)(1+o(1)) + O\Bigr(\sqrt{\,Tx\log x} +
  \frac{T}{x^2} (\log (kT))^2\Bigl)
  \nonumber\\ &
  = T (\log x)(1+o(1)),
 \end{align}
 as $x \to \infty,$
uniformly for
 $ x \leq T \leq e^x$ and $3 \leq k \leq \sqrt{x}/(\log x)^{3}$.
 In this range, we also have
   $ (\log T)(\log (kT))^2 = o(T \log x),$
   as $x \to \infty$.
 Equations \eqref{intlhsrhs} and \eqref{eqn42neww} imply that
\begin{equation}\nonumber
   F^+_{\chi_{_\square}}(x,T)  \sim  \frac{T }{2\pi} \log x,
\end{equation}
as $x \to \infty,$
uniformly for $ x \leq T \leq e^x$ and $3 \leq k \leq \sqrt{x}/(\log x)^{3}.$
\end{proof}

\section{Proof of Theorem \ref{least-under-pc}}
We start by introduce some notation and proving some useful lemmas.
Define \begin{align}
\label{Sigma-plus-def}
\Sigma^+_{\chi_{_\square}}(x,T,v)
\coloneq
\sum_{0 \le  \gamma \le T} x^{i\gamma}e^{iv\gamma}.
\end{align}
An easy computation gives
$W(u) = \int_{-\infty}^{\infty} e^{ivu}e^{-2 \vert v \vert }\, \dx v$ and
\begin{equation*}
F^+_{\chi_{_\square}}(x,T) = 
\int_{-\infty}^{\infty} \vert \Sigma^+_{\chi_{_\square}}(x,T,v) \vert ^2\,e^{-2 \vert v \vert }\, \dx v,
\end{equation*}
implying that $F^+_{\chi_{_\square}}(x,T)\ge 0$.

\begin{lem}\label{lemlpq}
For $x\ge 2$ and $T\ge 0$, we have 
\begin{equation*}
\int_{x}^{2x} \vert \Sigma^+_{\chi_{_\square}}(t,T,0)  \vert^2 \, \dx t 
\ll 
 x F^+_{\chi_{_\square}}(x, T).  
\end{equation*}
\end{lem}
\begin{proof} Note that
\begin{align*}
\int_{x}^{2x} \vert\Sigma^+_{\chi_{_\square}}(t,T,0) \vert^2 \, \dx t 
&= x \int_{0}^{\log 2} \vert\Sigma^+_{\chi_{_\square}}(x,T,v) \vert^2 e^{v} \, \dx v \nonumber\\
&\ll 
x \int_{-\infty}^{\infty}  \vert\Sigma^+_{\chi_{_\square}}(x,T,v) \vert^2 e^{-2 \vert v \vert } \, \dx v
= 
x F^+_{\chi_{_\square}}(x, T).
\qed
\end{align*}
\renewcommand{\qed}{}
\end{proof}

\begin{lem}\label{lem3}
For $x\ge 2$ and $T > U \ge 0$, we have 
\[
\vert
\Sigma^+_{\chi_{_\square}}(x,T,0) - \Sigma^+_{\chi_{_\square}}(x,U,0)
\vert
\ll 
\sqrt{\, T \max_{U \le  t \le T} F^+_{\chi_{_\square}}(x,t)}.
\]
\end{lem}

\begin{proof}
Using \eqref{Sigma-plus-def}, we define
\begin{equation}
\label{defgv}
h(x, T, U, v) = 
\Sigma^+_{\chi_{_\square}}(x,T,v) -
\Sigma^+_{\chi_{_\square}}(x,U,v) 
= \sum_{\substack{U <  \gamma \le  T
}} x^{i\gamma}e^{i\gamma v}.
\end{equation}

\noindent
Observe that

\begin{align}
\nonumber
   \int_{-\infty}^{\infty}
    \vert h(x, T, U, v) \vert ^2e^{-2 \vert v \vert }& \, \dx v 
   =
   \sum_{\substack{U <  \gamma_j \le  T\\ j = 1,2}} 
   x^{i(\gamma_1-\gamma_2)}
   \int_{-\infty}^{\infty}
   e^{i(\gamma_1-\gamma_2)v}e^{-2 \vert v \vert } \, \dx v 
  \\
   \label{G-F-relation}
   & =
%   \notag
    \sum_{\substack{U <  \gamma_j \le  T\\
   j = 1,2
   %\\L(\frac{1}{2} + i\gamma, \chi_j) = 0
   }} 
   x^{i(\gamma_1-\gamma_2)}
   W(\gamma_1 - \gamma_2)
 %  \\&
   = F^+_{\chi_{_\square}}(x,T) - F^+_{\chi_{_\square}}(x,U).
\end{align}
Letting $H(v) =  \vert  h(x, T, U, v) \vert ^2$,
we first remark that, using \eqref{G-F-relation}, we have
\begin{equation}
\label{G-estim}
\int_{-1}^1 
\vert  H(v)  \vert 
\, \dx v
\ll
\int_{-\infty}^{\infty}
\vert  H(v)  \vert 
\,e^{-2 \vert v \vert }\, \dx v
\ll
F^+_{\chi_{_\square}}(x,T) + F^+_{\chi_{_\square}}(x,U).
\end{equation}
Combining the Sobolev--Gallagher inequality (Lemma \ref{Sob-Gal-ineq})
with
\eqref{G-estim} yields
\begin{align}
\notag
\vert h(x, T, U, 0) \vert ^2
= 
H(0) 
&
\ll \int_{-1}^1  \vert  H(v)  \vert  \, \dx v +  \int_{-1}^1  \vert  H^{\prime}(v)  \vert  \, \dx v  
\\
&
\label{Sob-Gall-ineq}    
\ll F^+_{\chi_{_\square}}(x,T) +  F^+_{\chi_{_\square}}(x,U) + \int_{-1}^1   \vert  H^{\prime}(v)  \vert  \, \dx v.
\end{align}
By the Cauchy--Schwarz inequality \eqref{Cauchy}, 
we obtain 
\begin{equation}
\!\!
\label{G-prime-first}
\int_{-1}^1   \vert  H^{\prime}(v)  \vert  \, \dx v  \ll  
\Bigl(\int_{-1}^1 
\vert  H(v)  \vert 
\, \dx v
\Bigr)^{1/2} 
\Bigl(\int_{-1}^1 \bigl \vert 
\sum_{U <  \gamma \le  T} \!\!
\gamma x^{i\gamma}e^{iv\gamma}\bigr \vert ^2\, \dx v
\Bigr)^{1/2}.
\end{equation}
Moreover, by partial summation we have
\[
\sum_{U <  \gamma \le  T}\gamma x^{i\gamma}e^{iv\gamma}
=
T h(x, T, U, v)
- 
\int_{U}^T h(x, t, U, v)\, \dx t
\]
and, again using \eqref{G-F-relation}, we obtain
\begin{align}
\notag
\int_{-1}^{1}
\bigl \vert 
\sum_{U <  \gamma \le  T} & \gamma x^{i\gamma}e^{iv\gamma}\bigr \vert ^2 \, \dx v
\ll 
\int_{-\infty}^{\infty}
\bigl \vert 
\sum_{U <  \gamma \le  T}\gamma x^{i\gamma}e^{iv\gamma}\bigr \vert ^2\,e^{-2 \vert v \vert }\, \dx v
\\
\notag
&\ll 
T^2
\int_{-\infty}^{\infty} H(v)   \,e^{-2 \vert v \vert }\, \dx v
+ 
T  \int_{U}^T 
 \int_{-\infty}^{\infty} \vert h(x, t, U, v) \vert ^2  \,e^{-2 \vert v \vert }\, \dx v \,\, \dx t
\\
\notag
& 
\ll
T^2  (F^+_{\chi_{_\square}}(x,T) + F^+_{\chi_{_\square}}(x,U))
+
T  \int_{U}^T  (F^+_{\chi_{_\square}}(x,t) + F^+_{\chi_{_\square}}(x,U)) \,\, \dx t
 \\
\label{G-prime-second}  
&
\ll T^2  \max_{U\le t\le T}F^+_{\chi_{_\square}}(x,t).
\end{align}
Combining  \eqref{G-estim} and \eqref{G-prime-first}-\eqref{G-prime-second}
we obtain the estimate
\begin{equation}
\label{thesis}
\int_{-1}^1  \vert H^{\prime}(v) \vert \, \dx v  \ll T\max_{U\le t\le T}F^+_{\chi_{_\square}}(x,t),
\end{equation}
and the lemma  follows from \eqref{defgv}, \eqref{Sob-Gall-ineq} and \eqref{thesis}.
\end{proof}

\begin{Rem}\label{rem1}
Since $\chi_{_{\square}}$ is quadratic, it is a real character and so 
$L(1/2+i\gamma, \chi_{_{\square}})=0$ if and only if $L(1/2-i\gamma,\chi_{_{\square}})=0$.
This implies the identity
\[
 \sum_{\substack{-T \le  \gamma \le  -U
%\\L(\frac{1}{2} + i\gamma, \chi) = 0
}} \frac{w^{i\gamma}}{1/2 + i\gamma} 
=
\overline{
   \sum_{\substack{U\le  \gamma \le  T }} 
\frac{w^{i\gamma}}{1/2 + i\gamma}},
\] 
valid for  $w\in \mathbb{R}$ and $T  \ge  U \ge0$.
\end{Rem}

\begin{lem}\label{lem21q}
Let 
\[
I(x,T) 
\coloneq 
 \sum_{\substack{ \vert \gamma \vert  \le  T 
}} \frac{x^{1/2+ i\gamma}}{1/2 + i\gamma}.
\]
Let $k > x^{\eps}$ and $J$ be the maximal integer such that $\log k \le   (x/2^J) (\log (x/2^J))^{-1} $.
Assuming RH for the Dirichlet $L$-function associated with the quadratic character modulo $k$ and
Hypothesis \ref{PCquadratic} in the ranges $x^{\eps} < k \le \exp\bigl( x(\log x)^{-1} \bigr)$ and
 $ x^{\eps} \le T \le x$,
we have for $j=0,\dots,J$,
\[
\Bigl \vert  I\bigl(\frac{x}{2^j},T\bigr)- I\bigl(\frac{x}{2^{j+1}},T\bigr)\Bigr \vert
\ll \sqrt{\frac{x}{2^j} \log (kx)} \, x^{\eps} ,
\]
uniformly for $
x^{\eps}
\le T \le  x/2^j$ and  $x^{\eps} < k \le \exp\bigl((x/2^j) (\log (x/2^j))^{-1}\bigr)$.
\end{lem}

\begin{proof} 
We first focus on the contribution of the non-negative 
imaginary parts in the zero-sum embedded into $I(x,T)$.

Let $U = x^{\eps}$.
Let $y$ be of the form $x/2^{j+1}$ with $j$ chosen from $0,\dotsc,J$.
We first handle the high-lying zeros.
Using partial summation, Hypothesis \ref{PCquadratic} and Lemma \ref{lem3}, for $c=1,2$ we obtain
\begin{align*}
\notag
  \Big \vert  \sum_{\substack{U <  \gamma \le  T}
  %\\L(\frac{1}{2} + i\gamma, \chi) = 0
  } \frac{(cy)^{ i\gamma}}{1/2 + i\gamma} \Big \vert  
 &
   \ll  \frac{\vert \Sigma^+_{\chi_{_\square}}(cy,T,0) - \Sigma^+_{\chi_{_\square}}(cy,U,0) \vert}{T}
 \\& \hskip3cm
\notag
  +
   \int_{U}^{T} 
   \vert
   \Sigma^+_{\chi_{_\square}}(cy,w,0) - \Sigma^+_{\chi_{_\square}}(cy,U,0)
   \vert
  \, \frac{\dx w}{w^2} \\
  \nonumber
  & \ll \frac{1}{\sqrt{T}} \Bigl(
\max_{U \le  t \le T}
F^+_{\chi_{_\square}}(cy,t) \Bigr)^{1/2} +
\int_{U}^{T}  \Bigl(
\max_{U \le  t \le w}
F^+_{\chi_{_\square}}(cy,t) \Bigr)^{1/2}\, \frac{\dx w}{w^{3/2}}
\\
&\ll \sqrt{\log (ky)} \log (T/U) 
\ll \sqrt{\log (kx)} \log x
\end{align*}
as $x \to \infty.$
Combining this with the equivalence stated in Remark \ref{rem1}, we have, for $U < T \le  y$, we have
  \begin{align}\label{treat1q}
\Big \vert   \sum_{\substack{ U <  |\gamma| \le  T}
 %\\L(\frac{1}{2} + i\gamma, \chi) = 0
 } \frac{(2y)^{1/2+ i\gamma}-y^{1/2+i\gamma}}{1/2 + i\gamma} \Big \vert   
 \ll \sqrt{y\log (kx)} \, \log x,
   \end{align}
as $x \to \infty$.

We now handle the low-lying zeros.
Recall that $y=x/2^{j+1}$ for some $0\le j\le J$. 
Using  Hypothesis \ref{PCquadratic} and Lemma \ref{lemlpq}, we have
\begin{align}
\notag
\Big \vert  \sum_{\substack{ 0 \le \gamma  \le  U
 }} 
 \frac{(2y)^{1/2+ i\gamma}-y^{1/2+i\gamma}}{1/2 + i\gamma} \Big \vert   
 & =
\Big \vert \sum_{\substack{ 0 \le \gamma \le  U
 }} 
 \int_{y}^{2y} w^{-1/2+i\gamma}\, \dx w \Big \vert  
\\& \notag
=
 \Big \vert \int_{y}^{2y} w^{-1/2} 
 \Sigma^+_{\chi_{_\square}}(w,U,0)
\, \dx w\Big \vert 
\\ &
\notag
 \ll
 \Big( \int_{y}^{2y} 
 \vert\Sigma^+_{\chi_{_\square}}(w,U,0) \vert^2
\, \dx w\Big)^{1/2}
 \ll 
  \sqrt{y F^+_{\chi_{_\square}}(y, U)} 
  \\
 \label{treat12q}
 &
 \ll 
 \sqrt{y U \log (ky)}
 \ll
 \sqrt{y\log (kx)} \, x^{\eps} 
\end{align}
as $x \to \infty$, where we have used the Cauchy--Schwarz inequality \eqref{Cauchy}
in the previous estimate.

Noting that the  possible zero at $1/2$
gives rise to a contribution $2(\sqrt{2}-1)\sqrt{y}$,
and
combining this with \eqref{treat1q}-\eqref{treat12q},
we obtain 
\begin{equation*}
%\label{final-interval}
\Bigl\vert \sum_{\vert \gamma \vert \le  T}
\frac{(2y)^{1/2+ i\gamma}-y^{1/2+i\gamma}}{1/2 + i\gamma}
\Bigr\vert
\ll  \sqrt{y \log (kx)} \, x^{\eps}  ,
\end{equation*}
and this completes the proof.
\end{proof}

We are now ready to prove Theorem \ref{least-under-pc}.
\begin{proof}[Proof of Theorem \ref{least-under-pc}]
Let $q$ be a prime number.
Assuming RH for the Dirichlet $L$-function associated with the quadratic character modulo $q$, equation \eqref{explform-all-gammas} implies that
\begin{equation}
\label{explform-all-gammasq}
\theta(x,\chi_{_\square}) = - \sum_{\substack{ \vert \gamma \vert  \le  Z 
}} \frac{x^{1/2 + i\gamma}}{1/2 + i\gamma} +
O\Bigl(\sqrt{x} + \frac{x}{Z} \bigl(\log (qxZ)\bigr)^2\Bigr).
\end{equation}
Let $Z= \sqrt{x}(\log (qx))^2.$ Using Lemma \ref{lem21q} and \eqref{explform-all-gammasq}, we have, for $x^{\eps} < q \le \exp\bigl( x(\log x)^{-1} \bigr)$
\begin{align*}
%\label{thetaxchipenal}
\theta(x, \chi_{_\square})
&= \sum_{j=0}^{J} \Bigl(
\theta\bigl(\frac{x}{2^j}, \chi_{_\square} \bigr)
- 
\theta\bigl(\frac{x}{2^{j+1}}, \chi_{_\square}\bigr)
\Bigr) +
\theta\bigl(\frac{x}{2^{J+1}}, \chi_{_\square} \bigr) \nonumber \\ \nonumber
& \ll x^{1/2+\eps} \sqrt{\log (qx)}  + \sqrt{x} + \frac{x (\log (qxZ))^2}{Z} + \frac{x}{2^J} \\
% \nonumber
% & \ll
% x^{1/2} (\log (qx))^{1/2} x^{\eps} 
% + x^{1/2} + \frac{x (\log (qxZ))^2}{Z} 
% + (\log q) (\log x)  \\
& \ll
x^{1/2+\eps} \sqrt{\log (qx)} + (\log q) (\log x).
\end{align*}
For $q \le x^{\eps}$, again assuming RH for the Dirichlet $L$-function associated with the quadratic character modulo $q$, from Lemma \ref{davenlem3}
we obtain
\begin{align} \nonumber
\theta(x,\chi_\square) \ll \sqrt{x} (\log (qx))^2,
\end{align}
Combination of the previous two estimates, leads for $q \le \exp\bigl( x(\log x)^{-1} \bigr)$ to the upper bound
\begin{align}\label{thetaxchi}
\theta(x,\chi_\square) \ll x^{1/2+\eps} \sqrt{\log (qx)} + (\log q) (\log x).
\end{align}

Assume that all primes less than $c_1 x$ are quadratic residues modulo $q$, where $c_1>1$ is a  constant.
This implies that $\chi_{_\square}(p) = 1$
for $p < c_1 x$.
Goal is to deduce an upper bound on $x$.

From Chebyshev's bound we have
\[
\theta(c_1 x, \chi_{_\square}) > \theta(x,\chi_{_\square}) = \theta(x) \ge c_2 x,
\]
for some $c_2>0$, independent from $c_1$.
Combining this estimate with \eqref{thetaxchi} leads for $q \le \exp\bigl( x(\log x)^{-1} \bigr)$ and $x$ sufficiently large to
$$
c_2 x < \theta(c_1 x, \chi_{_\square}) \ll
x^{1/2+\eps} \sqrt{\log (qx)} + (\log q) (\log x)$$
It follows that $ x \ll (\log q)^{1+2\eps}$, implying the required result.
\end{proof}

\begin{Rem}
 Observe that in the previous proof we chose $\sqrt{x} (\log (qx))^2$ as the upper bound for the height of the non-trivial zeros, denoted by $Z$. Therefore, it is enough to assume that  Hypothesis \ref{PCquadratic} holds up to the maximum value $T = \sqrt{x} (\log (qx))^2$. However, in the hypothesis of Lemma \ref{lem21q} we allow $T$ to range up to $x$. We retain this larger range to keep the statement consistent with Lemma \ref{lem21} and \cite[Lemma 9]{mathannalen}.
\end{Rem}

\section{Proofs of Theorems \ref{mor5} and \ref{least-prime-gpc}}

We first record the following lemma that relates the sum over zeros in the lower half-plane to the conjugate sum over zeros in the upper half-plane. This relationship follows from the fact that $L(1/2 + i\gamma, \chi) = 0$ if and only if $L(1/2 - i\gamma, \overline{\chi}) = 0$.

\begin{lem}\label{conj-trick-lemma}\cite[Lemma\,8]{mathannalen}
If $(a, k)= 1$, $w\in \mathbb{R}$ and $T\ge0$, then
\[
\sum_{\chi \pmod{k}}\overline{\chi}(a) \sum_{\substack{-T \le  \gamma \le  0
\\L(1/2 + i\gamma, \chi) = 0
}} \frac{w^{i\gamma}}{1/2 + i\gamma} 
=
\overline{
    \sum_{\chi \pmod{k}}
  \overline{\chi}(a) \!\!\! 
   \sum_{\substack{0\le  \gamma \le  T 
 \\ L(1/2+i\gamma,\chi)= 0
   }} 
\frac{w^{i\gamma}}{1/2 + i\gamma}}.
\] 
\end{lem}

From now on, whenever we write $\gamma$ and $\gamma_j$ in the summation 
without additional specifications, we assume, that $L(1/2 + i\gamma, \chi) = 0$,
respectively
$L(1/2 + i\gamma_j, \chi_ j) = 0$. Moreover, the zeros are counted with multiplicity.
For fixed coprime $a$ and $k$, we
define the following weighted sums over the zeros:
\begin{equation*}
%\label{Sigma-plus-def}\notag
\Sigma^+(x,T,v) 
=
\Sigma^+(x,T,v;k,a) 
\coloneq
\sum_{\chi \pmod{k}}
\overline{\chi} (a)
\sum_{0 \le  \gamma \le T} x^{i\gamma}e^{iv\gamma},
\end{equation*}
and
\begin{equation*}
%\label{Sigma-def}\notag
\Sigma(x,T,v) 
=
\Sigma(x,T,v;k,a) 
\coloneq
\sum_{\chi \pmod{k}}
\overline{\chi} (a)
\sum_{0 \le \vert \gamma \vert \le T} x^{i\gamma}e^{iv\gamma}.
\end{equation*}

We now state two lemmas that connect $\Sigma$ and $\Sigma^+$
with $F_{k}(x,T)$ and $F_k^+(x,T)$.

\begin{lem}\label{lem3-mathannalen}\cite[Lemma\,6]{mathannalen}
For $x\ge 2$ and $T > U \ge 0$, we have 
\[
\vert
\Sigma^+(x,T,0) - \Sigma^+(x,U,0)
\vert
\ll 
\sqrt{\, T \max_{U \le  t \le T} F^{+}_{k}(x,T)}.
\]
\end{lem}

\begin{lem}\label{lemlp-mathannalen}\cite[Lemma\,7]{mathannalen}
For $x\ge 2$ and $T\ge 0$, we have 
\begin{equation*}
\int_{x}^{2x} \vert\Sigma(t,T,0) \vert^2 \, \dx t 
\ll 
 x F_k(x, T)
\quad 
\textrm{and}
\quad
\int_{x}^{2x} \vert \Sigma^+(t,T,0)  \vert^2 \, \dx t 
\ll 
 x F^+_k(x, T).  
\end{equation*}
\end{lem}

We now prove the following lemma that will play a crucial role in the proof 
of Theorem \ref{mor5}.

\begin{lem}\label{lem21}
Let $c_1>0$ and $c_2\in(1/2,1)$ be fixed.
Let further $1\le a \le k-1, (a, k)= 1$ and put
\[
I(x,k,a,T) 
\coloneq 
\frac{1}{\varphi(k)} 
\sum_{\substack{\chi \pmod{k}}} \overline{\chi}(a) \sum_{\substack{ \vert \gamma \vert  \le  T 
%\\L(\frac{1}{2} + i\gamma, \chi) = 0
}} \frac{x^{1/2+ i\gamma}}{1/2 + i\gamma}.
\]
Let $k > \exp\bigl(c_1(\log x)^{c_2}\bigr)$
and $J$ be the maximal integer such that
$k\le (x/2^J)$ $\exp\bigl(- 2c_1 (\log (x/2^J))^{c_2}\bigr)$.
Assuming GRH and Hypothesis \ref{conjgpc2} in the ranges 
$$ \exp\bigl(c_1(\log x)^{c_2}\bigr)
< k \le x \exp\bigl(- 2c_1 (\log x)^{c_2}\bigr)\quad\text{~and~}\quad \exp\bigl(c_1(\log x)^{c_2}\bigr)
\le  T \le x,$$
we have for $j=0,\dots,J$, that
\[
\Bigl \vert  I\bigl(\frac{x}{2^j},k,a,T\bigr)- I\bigl(\frac{x}{2^{j+1}},k,a,T\bigr)\Bigr \vert 
\ll \sqrt{\frac{x}{2^j{\varphi(k)}}}\, \exp\bigl(c_1(\log x)^{c_2}\bigr)
\]
uniformly for
$ \exp\bigl(c_1(\log x)^{c_2}\bigr)
< k\le (x/2^j)\exp\bigl(-2c_1(\log (x/2^j))^{c_2}\bigr)$
and $ \exp\bigl(c_1(\log x)^{c_2}\bigr)
\le T \le  x/2^j$.
\end{lem}

\begin{proof} 
We first focus on the contribution of the non-negative 
imaginary parts in the zero-sum embedded into $I(x,k,a,T)$.

Let $U = \exp\bigl(c_1(\log x)^{c_2}\bigr)$.
Let $y$ be of the form $x/2^{j+1}$ with $j$ chosen from $0,\dotsc,J$
and $k\le y
\exp\bigl(- 2c_1 (\log y)^{c_2}\bigr)$
We first handle the high-lying zeros.
Using partial summation, Hypothesis \ref{conjgpc2} and Lemma
\ref{lem3-mathannalen}, for $c=1,2$ we obtain
\begin{align*}
\notag
  \Big \vert \sum_{\chi \pmod{k}}
  \overline{\chi}(a) \sum_{\substack{U <  \gamma \le  T} } 
  &\frac{(cy)^{ i\gamma}}{1/2 + i\gamma} \Big \vert  
  \nonumber
   \ll  \frac{\vert \Sigma^+(cy,T,0) - \Sigma^+(cy,U,0) \vert}{T}
   \\& \hskip3cm
\notag
  +
   \int_{U}^{T} 
   \bigl\vert
   \Sigma^+(cy,w,0) - \Sigma^+(cy,U,0)
   \bigr\vert
  \, \frac{\dx w}{w^2} \\
  \nonumber
  & \ll \frac{1}{\sqrt{T}} \Bigl(
\max_{U \le  t \le T}
F^+_k(cy,t) \Bigr)^{1/2} +
\int_{U}^{T} 
\Bigl(
\max_{U \le  t \le w}
F^+_k(cy,t) \Bigr)^{1/2} \frac{\dx w}{w^{3/2}}
\\
%\label{partial-summation}
&\ll  \sqrt{{\varphi(k)}} \, 
\exp\bigl(\frac{c_1}{2} (\log y)^{c_2}\bigr)
\log(T/U)
\\&
\ll  \sqrt{{\varphi(k)}} \,
\exp\bigl(c_1(\log x)^{c_2}\bigr)
\end{align*}
as $x \to \infty,$ uniformly for $k\le y \exp\bigl(- 2c_1 (\log y)^{c_2}\bigr)$.
This implies that for $U < T \le  y$, we have
  \begin{align}\label{treat1}
\Big \vert  \sum_{\chi \pmod{k}}\overline{\chi}(a) \sum_{\substack{ U <  \gamma \le  T}
 %\\L(\frac{1}{2} + i\gamma, \chi) = 0
 } \frac{(2y)^{1/2+ i\gamma}-y^{1/2+i\gamma}}{1/2 + i\gamma} \Big \vert   
\ll \sqrt{{\varphi(k)}y}
\exp\bigl(c_1(\log x)^{c_2}\bigr),
   \end{align}
as $x \to \infty,$ uniformly for  $k\le y \exp\bigl(- 2c_1 (\log y)^{c_2}\bigr)$.
From \eqref{treat1} (see Lemma \ref{conj-trick-lemma}) we also obtain
\begin{equation}
\label{negative-interval}
\Bigl\vert
\sum_{\chi \pmod{k}}\overline{\chi}(a) \sum_{-T \le  \gamma < -U}
\frac{(2y)^{1/2+ i\gamma}-y^{1/2+i\gamma}}{1/2 + i\gamma}
\Bigr\vert
\ll \sqrt{{\varphi(k)}y}
 \exp\bigl(c_1(\log x)^{c_2}\bigr).
\end{equation}

We now handle the low-lying zeros.
Recall that $y=x/2^{j+1}$ for some $0\le j\le J$. 
Using  Hypothesis \ref{conjgpc2} 
and Lemma \ref{lemlp-mathannalen},
we obtain
\begin{align}\label{treat12}
\Big \vert  \sum_{\chi \pmod{k}} & \overline{\chi}(a)
\sum_{\substack{ \vert \gamma \vert \le  U
 }} 
 \frac{(2y)^{1/2+ i\gamma}-y^{1/2+i\gamma}}{1/2 + i\gamma} \Big \vert   
 =
\Big \vert  \sum_{\chi \pmod{k}}\overline{\chi}(a) \sum_{\substack{ \vert \gamma \vert \le  U
 }} 
 \int_{y}^{2y} w^{-1/2+i\gamma}\, \dx w \Big \vert   
 \nonumber\\
& =
 \Big \vert \int_{y}^{2y} w^{-1/2} 
 \Sigma(w,U,0)
\, \dx w\Big \vert 
 \ll
 \Big( \int_{y}^{2y} 
 \vert\Sigma(w,U,0) \vert^2
\, \dx w\Big)^{1/2}
\nonumber\\&
\nonumber
\ll 
  \sqrt{y F_k(y, U)}
 \ll \sqrt{\varphi(k) y U 
 \exp\bigl(c_1(\log x)^{c_2}\bigr)}
 \\&
 \ll \sqrt{{\varphi(k)}y} \exp\bigl(c_1(\log x)^{c_2}\bigr) ,
\end{align}
as $x \to \infty$, where we have used the Cauchy--Schwarz inequality \eqref{Cauchy}
in the previous estimate.

Combining \eqref{treat1}-\eqref{treat12}, we obtain 
\begin{equation*}
%\label{final-interval}
\frac{1}{\varphi(k)} 
\Bigl\vert
\sum_{\chi \pmod{k}}\overline{\chi}(a) \sum_{\vert \gamma \vert \le  T}
\frac{(2y)^{1/2+ i\gamma}-y^{1/2+i\gamma}}{1/2 + i\gamma}
\Bigr\vert
\ll \sqrt{\frac{y}{{\varphi(k)}}} \exp\bigl(c_1(\log x)^{c_2}\bigr),
\end{equation*}
completing the proof.
\end{proof}

\begin{proof}[Proof of Theorem \ref{mor5}]

Under GRH, for $Z\le  x$ we argue as in \cite[eq. (58)]{mathannalen},
thus getting 
\begin{align}\label{eqn1}
   \psi(x;k,a) 
   & = \frac{1}{\varphi(k)} 
   \Bigl( x -
\sum_{\substack{  \chi 
\pmod{k}}} \overline{\chi}(a) \sum_{\substack{ \vert \gamma \vert  \le  Z 
%\ \L(\frac{1}{2} + i\gamma, \chi) = 0
}} \frac{x^{1/2 + i\gamma}}{1/2 + i\gamma} \Bigr) 
+ 
O\Bigl(
\frac{x (\log (kxZ))^2}{Z}
\Bigr),
   \end{align}
as $x \to \infty.$

 Let $k > \exp\bigl(c_1(\log x)^{c_2}\bigr)$
and let $J$ be the maximal integer such that
$(x/2^J)$ $\exp\bigl(-2c_1(\log (x/2^J))^{c_2}\bigr)$ $ \ge  k$.
Using  the Brun--Titchmarsh theorem (Classical Theorem \ref{BT-thm}), 
Lemma \ref{lem21} and equation \eqref{eqn1} with 
$ \exp\bigl(c_1(\log x)^{c_2}\bigr) 
< k\le x \exp\bigl(-2c_1(\log x)^{c_2}\bigr)$, 
$Z = x/2^j
$,
$j=0,\dotsc, J$, we deduce that, for every $1\le a \le k-1$, $(a,k)=1$,
\begin{align}\notag
\psi(x;k,a)- \frac{x}{\varphi(k)}
&= \sum_{j=0}^{J} \Bigl(
\psi\bigl(\frac{x}{2^j};k,a\bigr)
- 
\psi\bigl(\frac{x}{2^{j+1}};k,a\bigr) - \frac{x}{2^{j+1} \varphi(k)}
\Bigr) 
\\&\hskip2cm
\nonumber
+
\psi\bigl(\frac{x}{2^{J+1}};k,a\bigr) - \frac{x}{2^{J+1} \varphi(k)}
\\ \nonumber
& \ll
\sum_{j=0}^{J} 
\sqrt{\frac{x}{2^j {\varphi(k)}}}  \exp\bigl(c_1(\log x)^{c_2}\bigr) + (\log x)^2 +
\frac{x}{2^{J+1} \varphi(k)} \\
\label{dissection-estim}
& \ll \sqrt{\frac{x}{{\varphi(k)}}} 
\exp\bigl(c_1(\log x)^{c_2}\bigr).
\end{align}
Using \eqref{GRH-only}, we obtain
\begin{align}\label{GRH-only-condition-q}
\psi(x;k,a)- \frac{x}{\varphi(k)}
\ll \sqrt{x} (\log x)^2
\ll
\sqrt{\frac{x}{{\varphi(k)}}} \exp\bigl(c_1(\log x)^{c_2}\bigr),
 \end{align}
 for $k \le \exp\bigl(c_1(\log x)^{c_2}\bigr)$ and for every $1\le a \le k-1$, $(a,k)=1$.
The proof is now completed on combining the estimates \eqref{dissection-estim} and \eqref{GRH-only-condition-q}.
\end{proof}

\begin{Rem}
In fact, in the proof of Theorem \ref{mor5} we can choose 
$Z=\sqrt{\varphi(k)x}<\sqrt{kx}\le x \exp\bigl(-c_1(\log x)^{c_2}\bigr)$,
but this would not change very much the uniformity ranges needed
in Hypothesis \ref{conjgpc2}. That is why we decided to use the slightly larger choice $Z=x$
instead.
\end{Rem}

\begin{proof}[Proof of Theorem \ref{least-prime-gpc}]
Using the Brun--Titchmarsh theorem (Classical Theorem \ref{BT-thm}), we obtain
$$
\psi(x;k,a) - \theta(x;k,a) \ll 
 \pi(\sqrt{x}\,; k , a)\ \log x
\ll
\frac{\sqrt{x}\, \log x}{\varphi(k) \log (x/k)}
\ll
\frac{\sqrt{x}\, \log x}{\varphi(k)},
$$
for every $1\le a \le k-1$, $(a,k)=1$, and
uniformly for
$k \le x \exp\bigl(- 2c_1 (\log x)^{c_2}\bigr)$ as $x\rightarrow \infty$.
Theorem \ref{mor5} implies that  we also have, uniformly in 
$ 1 \le k \le x \exp\bigl(- 2c_1 (\log x)^{c_2}\bigr)$,
\begin{align*}
    \theta(x;k,a) &= \frac{x}{\varphi(k)} 
    + O\Bigl( \frac{\sqrt{x} \,\log x}{\varphi(k)} +  \sqrt{\frac{x}{{\varphi(k)}}}\, \exp\bigl(c_1(\log x)^{c_2}\bigr) \Bigr) \\
   & = \frac{x}{\varphi(k)} + O \Bigl(  \sqrt{\frac{x}{\varphi(k)}}\, \exp\bigl(c_1(\log x)^{c_2}\bigr) \Bigr). 
\end{align*}
In order to ensure that $\theta(x;k,a)$ is positive, we write the error term
with an explicit positive 
constant $D$ and solve
\[
\frac{x}{\varphi(k)} - D \sqrt{\frac{x}{\varphi(k)}}\exp\bigl(c_1(\log x)^{c_2}\bigr) > 0
\]
for $x$, which yields $x \gg \varphi(k) \exp\bigl(2c_1(\log k)^{c_2}\bigr)$.
This implies
\[
p(k,a)  \ll \varphi(k) \exp\bigl(2c_1(\log k)^{c_2}\bigr)
\]
for every $1\le a \le k-1$, $(a,k)=1$, completing the proof.
\end{proof}

\medskip
\noindent \textbf{Acknowledgment}. 
We wish to thank 
Professors Andrew Granville and Youness Lamzouri for helpful feedback.
The first and the third author would like to thank the Max-Planck-Institut f\"ur Mathematik in Bonn 
for the hospitality and excellent working conditions.

\medskip
\noindent \textbf{Data availability.} Not applicable.

\medskip
\noindent \textbf{Competing Interests.} The authors have no competing interests to declare that are relevant to the content of this article.

\end{document}